\documentclass[11pt]{amsart}
\usepackage[margin=1in]{geometry}

\usepackage{graphicx}
\usepackage[active]{srcltx}
\usepackage{amsmath}
\usepackage{amscd}
\usepackage{tikz-cd}
\usepackage{amssymb}
\usepackage[mathscr]{euscript}
\usepackage{enumerate}
\usepackage{hyperref}
\usepackage{times}

\DeclareMathOperator{\Hom}{Hom}

\DeclareMathOperator{\Ext}{Ext}

\DeclareMathOperator{\DD}{D}

\newcommand{\D}{\mathscr{D}}
\newcommand{\F}{\mathscr{F}}

\newcommand{\dR}{\mathrm{dR}}

\DeclareMathOperator{\grHom}{\sideset{^*}{}\Hom}

\DeclareMathOperator{\grD}{\sideset{^*}{}\DD}

\newcommand{\fm}{\mathfrak{m}}

\theoremstyle{plain}
\newtheorem{thm}{Theorem}[section]
\newtheorem*{thm*}{Theorem}
\newtheorem{bigthm}{Theorem}
\newtheorem{prop}[thm]{Proposition}
\newtheorem{lem}[thm]{Lemma}
\newtheorem{cor}[thm]{Corollary}

\theoremstyle{definition}
\newtheorem{definition}[thm]{Definition}

\theoremstyle{remark}
\newtheorem{remark}[thm]{Remark}
\newtheorem{example}[thm]{Example}

\newtheorem{question}[thm]{Question}

\numberwithin{equation}{thm}

\begin{document}
\title{A prime-characteristic analogue of a theorem of Hartshorne-Polini}
\author{Nicholas Switala \and Wenliang Zhang}
\address{Department of Mathematics, Statistics, and Computer Science \\ University of Illinois at Chicago \\ 322 SEO (M/C 249) \\ 851 S. Morgan Street \\ Chicago, IL 60607}
\email{nswitala@uic.edu, wlzhang@uic.edu}
\thanks{The first author gratefully acknowledges NSF support through grant DMS-1604503. The second author is partially supported by the NSF through grant DMS-1606414 and CAREER grant DMS-1752081.}

\begin{abstract}
Let $R$ be an $F$-finite Noetherian regular local ring containing its algebraically closed residue field $k$ of positive characteristic, and let $M$ be an $\F$-finite $\F$-module over $R$ in the sense of Lyubeznik (for example, any local cohomology module of $R$). We prove that the $\mathbb{F}_p$-dimension of the space of $\F$-module morphisms $M \rightarrow E(R/\fm)$ (where $\fm$ is any maximal ideal of $R$ and $E(R/\fm)$ is the $R$-injective hull of $R/\fm$) is equal to the $k$-dimension of the Frobenius stable part of $\Hom_R(M,E(R/\fm))$. This is a positive-characteristic analogue of a recent result of Hartshorne and Polini for holonomic $\D$-modules in characteristic zero. \end{abstract}

\maketitle

\section{Introduction}\label{intro}
In the study of finiteness properties of local cohomology modules there has been an emerging theme: the parallel between (holonomic) $\D$-modules in characteristic zero and ($\F$-finite) $\F$-modules in characteristic $p>0$ ({\it cf.} \cite{ZhangInjDim},\cite{SwitalaZhangInjDim}). This paper continues the line of research under the same theme: we prove $\F$-module analogues of results obtained by Hartshorne and Polini in \cite{HartshornePolini} and by the authors in \cite{SwitalaZhangGradedDual} for $\D$-modules over formal power series or polynomial rings. The theory of ($\F$-finite) $\F$-modules will be reviewed in the next section. We begin by recalling the results for holonomic $\D$-modules in characteristic zero.

\begin{thm}\label{top dR theorems}
\begin{enumerate}[(a)]
\item \cite[Corollary 5.2]{HartshornePolini},\cite[Theorem 5.1]{SwitalaThesis} Let $R = k[[x_1, \ldots, x_n]]$ where $k$ is a field of characteristic zero, let $\fm \subseteq R$ be the maximal ideal, and let $\D = \D(R,k)$ be the ring of $k$-linear differential operators on $R$. Denote by $E$ the $R$-module $H^n_{\fm}(R)$, which is an $R$-injective hull of $k$. If $M$ is a holonomic $\D$-module, then
\begin{align*}
\dim_k H^n_{\dR}(M) = \dim_k H^0_{\dR}(D(M)) &= \dim_k \Hom_{\D}(M, E)\\
&= \max\{t \in \mathbb{N} \mid \exists\ {\rm a\ } \D{\rm -module\ surjection}\ M \rightarrow E^t\},
\end{align*}
where $H^i_{\dR}(-)$ denotes the de Rham cohomology of a $\D$-module, and $D(-) = \Hom_R(-,E)$ is the \emph{Matlis dual} functor.
\item \cite[Theorem 5.3]{SwitalaZhangGradedDual} Let $R = k[x_1, \ldots, x_n]$ where $k$ is a field of characteristic zero, let $\fm \subseteq R$ be the irrelevant maximal ideal, and let $\D = \D(R,k)$ be the ring of $k$-linear differential operators on $R$. Denote by $E$ the $R$-module $H^n_{\fm}(R)$, which (with the correct choice of grading) is a graded $R$-injective hull of $k$. If $M$ is a finitely generated graded $\D$-module with finite-dimensional de Rham cohomology spaces (for example, a graded holonomic $\D$-module), then
\begin{align*}
\dim_k H^n_{\dR}(M) = \dim_k H^0_{\dR}(\grD(M)) &= \dim_k \Hom_{\D}(M, E)\\
&= \max\{t \in \mathbb{N} \mid \exists\ {\rm a\ } \D{\rm -module\ surjection}\ M \rightarrow E^t\},
\end{align*}
where $\grD(-) = \grHom_R(-,E)$ is the \emph{graded Matlis dual} functor.
\end{enumerate}
\end{thm}

In Theorem \ref{top dR theorems}(b), a \emph{graded} $\D$-module is a graded $R$-module on which the operators $\partial_j = \frac{\partial}{\partial x_j} \in \D$ act as graded $k$-linear maps of degree $-1$. If $I \subseteq k[[x_1, \ldots, x_n]]$ is an ideal (resp. $I \subseteq k[x_1, \ldots, x_n]$ is a homogeneous ideal), the local cohomology modules $H^i_I(R)$ are holonomic (resp. graded holonomic) $\D$-modules, and so Theorem \ref{top dR theorems} can be applied to them. 

The main result of this paper, Theorem \ref{big thm}, is an $\F$-module analogue of Theorem \ref{top dR theorems}. The finiteness condition analogous to the holonomicity of $\D$-modules is the $F$-finiteness of $\F$-modules, and local cohomology modules satisfy this condition. De Rham cohomology for $\D$-modules is ill-behaved in positive characteristic, so any analogue of Theorem \ref{top dR theorems} will require a replacement for $H^0_{\dR}(D(M))$. The desired replacement turns out to be the \emph{(Frobenius) stable part} of $D(M)$. (See section \ref{Fmodprelims} below for the relevant definitions.) The hypotheses of Theorem \ref{big thm} are more general than those of Theorem \ref{top dR theorems} and include both the cases of polynomial and formal power series rings; observe that there is no need to state and prove a graded version separately.

\begin{bigthm}[Corollary \ref{special case of main theorem}]\label{big thm}
Let $(R,\fm)$ be a Noetherian regular local ring containing its algebraically closed residue field $k$ of characteristic $p > 0$, and let $M$ be an $\F$-finite $\F$-module over $R$. Let $E = E_R(R/\fm)$ be the $R$-injective hull of $R/\fm$, and denote by $D(-)$ the exact functor $\Hom_R(-,E)$ on the category of $R$-modules. Then
\begin{align*}
\dim_{\mathbb{F}_p} \Hom_{\F}(M, E) &= \dim_k D(M)_s \\
&= \max\{t \in \mathbb{N} \mid \exists\ {\rm an\ } \F{\rm -module\ surjection}\ M \rightarrow E^t\},
\end{align*}
where $(-)_s$ denotes the \emph{stable part} of a Frobenius module, and $\Hom_{\F}$ denotes the $\mathbb{F}_p$-space of $\F$-module morphisms.
\end{bigthm}

In section \ref{Fmodprelims}, we collect the necessary preliminary material on Frobenius modules and $\F$-modules. Much of this section recalls definitions and results due to Hartshorne and Speiser, Lyubeznik, and Bhatt and Lurie, but Proposition \ref{Frobenius structure on Hom} appears to us to be new. Theorem \ref{big thm} is proved in section \ref{mainthm}.

\subsection*{Acknowledgments} The first author thanks Rankeya Datta, Linquan Ma, and Mircea Musta\c{t}\u{a} for helpful discussions. The authors are grateful to the referee for helpful comments and suggestions.

\section{Frobenius modules and $\F$-modules}\label{Fmodprelims}

We begin with some notation and conventions. All rings considered in this paper are commutative with identity $1$. Except in section \ref{intro}, all rings are of characteristic $p > 0$ unless otherwise noted. Throughout this section, $A$ denotes such a ring, and we reserve the letter $R$ for \emph{regular} Noetherian rings; we will repeat these assumptions in the hypotheses of definitions and theorems. All local rings are assumed to be Noetherian. 


We denote by $F$ (or $F_A$ if the context demands) the \emph{Frobenius} endomorphism $F: A \rightarrow A$ defined by $F(a) = a^p$ for all $a \in A$. If $M$ is an $A$-module, we can consider the $A$-modules $F^*M$ and $F_*M$. The $A$-module $F_*M$ has the same underlying Abelian group as $M$, with $A$-action defined by $a \ast m = a^pm$. On the other hand, as an Abelian group, $F^*M$ can be expressed as $F_*A \otimes_{A} M$, where the notation means that we form the tensor product by regarding $A$ as a right $A$-module via the Frobenius. Explicitly, for $a, b \in A$ and $m \in M$, we have $a(b \otimes m) = ab \otimes m$ and $a \otimes bm = ab^p \otimes m$. 

The following is just the well-known adjunction between restriction and extension of scalars; we record it here separately so as to have a specific reference for the formulas in the sequel.
\begin{prop}\label{adjunction}
Let $A$ be a ring of characteristic $p > 0$, and let $M$ be an $A$-module. There is a bijective correspondence between $A$-linear maps $M \rightarrow F_*M$ and $A$-linear maps $F^*M \rightarrow M$.
\end{prop}

\begin{proof}
If $\varphi: M \rightarrow F_*M$ is an $A$-linear map, the corresponding $A$-linear map $\psi: F^*M = F_*A \otimes_{A} M \rightarrow M$ is defined by $\psi(a \otimes m) = a\varphi(m)$. Conversely, if $\psi: F^*M = F_*A \otimes_{A} M \rightarrow M$ is an $A$-linear map, the corresponding $A$-linear map $\varphi: M \rightarrow F_*M$ is defined by $\varphi(m) = \psi(1 \otimes m)$.
\end{proof}

The main objects of this paper are $A$-modules equipped with $A$-linear maps to or from their pushforwards and pullbacks along the Frobenius $F_A$. \emph{Frobenius modules} over $A$ (that is, $A$-modules $M$ equipped with a choice of $A$-linear map $M \rightarrow F_*M$) were studied by Hartshorne and Speiser in \cite{HartshorneSpeiser} and, more recently, by Bhatt and Lurie in \cite{BhattLurie}. On the other hand, \emph{$\F$-modules} (that is, $A$-modules $M$ equipped with a choice of $A$-linear \emph{isomorphism} $M \rightarrow F^*M$), also known as \emph{unit Frobenius modules}, were introduced by Lyubeznik in \cite{LyubeznikFMod} and studied further by Emerton and Kisin in \cite{EmertonKisin} and Bhatt and Lurie in \cite{BhattLurie}, in the case where $A$ is regular and Noetherian. (By a celebrated theorem of Kunz \cite[Theorem 2.1]{KunzRegular}, the functor $F^*$ is \emph{exact} under these hypotheses, and this exactness is crucial to the theory of $\F$-modules.)

We now proceed to give the basic definitions and relationships between these objects.

\begin{definition}\label{Frobenius module} 
Let $A$ be a ring of characteristic $p > 0$. A \emph{Frobenius module} over $A$ is a pair $(M, \varphi_M)$ where $M$ is an $A$-module and $\varphi_M: M \rightarrow F_*M$ is an $A$-linear map. (When there is no danger of confusion, we sometimes write $\varphi$ for $\varphi_M$; we also sometimes refer simply to $M$ as a ``Frobenius module''.)
\end{definition}

An $A$-linear map $M \rightarrow F_*M$ is the same thing as an additive map $\varphi_M: M \rightarrow M$ such that $\varphi_M(am) = a^p\varphi_M(m)$ for all $a \in A$ and $m \in M$. In particular, the iterates $\varphi_M^i$ for $i \geq 0$ make sense. Frobenius modules over $A$ form a category $\mathrm{Mod}_A^{\mathrm{Fr}}$, where a morphism $(M, \varphi_M) \rightarrow (N, \varphi_N)$ is an $A$-linear map $f: M \rightarrow N$ such that $\varphi_N \circ f = F_*f \circ \varphi_M$. We denote by $\Hom_{A[F]}(M,N)$ the $\mathbb{F}_p$-space of Frobenius module morphisms $(M, \varphi_M) \rightarrow (N, \varphi_N)$. The reason for the notation is that a Frobenius module over $A$ is the same thing as a left module over the non-commutative ring $A[F]$ generated over $A$ by the symbol $F$, subject to the relations $Fa = a^pF$ for all $a \in A$.

\begin{definition}\label{stable parts}
Let $A$ be a ring of characteristic $p>0$, and let $(M, \varphi_M)$ be a Frobenius module over $A$.
\begin{enumerate}[(a)]
\item The (Frobenius) \emph{fixed part} $M^{\varphi = 1}$ of $M$ is the $\mathbb{F}_p$-subspace $\{m \in M \mid \varphi_M(m) = m\} \subseteq M$.
\item Suppose that $A$ contains a perfect field $k$ of characteristic $p > 0$. The (Frobenius) \emph{stable part} $M_s$ of $M$ is the $k$-subspace $\cap_{i \geq 0} \varphi_M^i(M) \subseteq M$.
\end{enumerate}
\end{definition}

If $k$ is any field of characteristic $p > 0$, the only solutions $\lambda \in k$ to the equation $\lambda^p = \lambda$ are the elements of the prime subfield $\mathbb{F}_p \subseteq k$. Therefore the fixed part $M^{\varphi = 1}$ can only be an $\mathbb{F}_p$-subspace. If $k$ is perfect, the iterated images $\varphi_M^i(M)$ are $k$-subspaces of $M$, and so the same is true for $M_s$.

It is clear that $M^{\varphi = 1} \subseteq M_s$. Under stronger hypotheses on $k$, we can say something more about the relationship between the fixed and stable parts:

\begin{prop}\label{fixed versus stable part}\cite[Exp. XXII, Corollaire 1.1.10]{SGA7-2}
Let $A$ be a ring containing an algebraically closed field $k$ of characteristic $p > 0$. Let $(M,\varphi_M)$ be a Frobenius module over $A$. If $M_s$ is a finite-dimensional $k$-space, then there is an isomorphism
\[
k \otimes_{\mathbb{F}_p} M^{\varphi = 1} \cong M_s
\]
of $k$-spaces. In particular, there exists a $k$-basis $\{m_1, \ldots, m_l\}$ of $M_s$ such that $\varphi_M(m_i) = m_i$ for $i = 1, \ldots, l$.
\end{prop}

Proposition \ref{fixed versus stable part} implies that if $M_s$ is a finite-dimensional $k$-space, then $M^{\varphi = 1}$ is a finite-dimensional $\mathbb{F}_p$-space (of the same dimension). The converse, however, is not true, as the following example shows.

\begin{example}\label{stable part too big}
Let $R = k[x]$ where $k$ is a perfect field of characteristic $p > 0$, let $F_R$ be the Frobenius endomorphism on $R$, and consider the \emph{perfection} $R^{1/p^{\infty}} = k[x^{1/p^{\infty}}]$, that is, the colimit
\[
\varinjlim(R \xrightarrow{F_R} R \xrightarrow{F_R} R \xrightarrow{F_R} \cdots).
\]
The Frobenius endomorphism $F_{R^{1/p^{\infty}}}$ is bijective. Therefore, if we regard $(R^{1/p^{\infty}}, F_{R^{1/p^{\infty}}})$ as a Frobenius module over $R$, we have $(R^{1/p^{\infty}})_s = R^{1/p^{\infty}}$, which is not a finite-dimensional $k$-space. However, $(R^{1/p^{\infty}})^{\varphi = 1}$ is simply $\mathbb{F}_p$, a one-dimensional $\mathbb{F}_p$-space. This is true whether or not $k$ is algebraically closed.
\end{example}

The following result of Hartshorne and Speiser provides one useful case in which the finiteness of the stable part is known (and so Proposition \ref{fixed versus stable part} applies).

\begin{thm}\label{Artinian stable part finite}\cite[Theorem 1.12]{HartshorneSpeiser}
Let $A$ be a local ring containing its perfect residue field $k$ of characteristic $p > 0$, and let $(M,\varphi_M)$ be a Frobenius module over $A$. If $M$ is an Artinian $A$-module, then $M_s$ is a finite-dimensional $k$-space, and the induced map $\varphi: M_s \rightarrow M_s$ is bijective.
\end{thm} 

If $(M, \varphi_M)$ is a Frobenius module over $A$, we have by Proposition \ref{adjunction} an $A$-linear map $\psi_M: F^*M \rightarrow M$. When $\psi_M$ is an isomorphism, $M$ is called a \emph{unit} Frobenius module by some authors \cite{EmertonKisin, BhattLurie}. In Lyubeznik's \cite{LyubeznikFMod}, which deals only with the case of a regular Noetherian ring $R$, unit Frobenius modules over $R$ are called \emph{$\F$-modules} (or $\F_R$-modules). We will follow Lyubeznik's notation and terminology.

\begin{definition}\label{F-module}
Let $R$ be a regular Noetherian ring of characteristic $p > 0$. An \emph{$\F$-module} over $R$ (or $\F_R$-module) is a pair $(M, \theta_M)$ where $M$ is an $R$-module and $\theta_M: M \xrightarrow{\sim} F^*M$ is an $R$-module isomorphism, called the \emph{structure morphism}. (When there is no danger of confusion, we sometimes refer simply to $M$ as an ``$\F$-module''.)
\end{definition}

Of course, if $(M, \theta_M)$ is an $\F_R$-module, then $(M, \varphi_M)$ is a Frobenius module over $R$, where $\varphi_M: M \rightarrow F_*M$ is the $R$-linear map that corresponds via adjunction to $\psi_M = \theta_M^{-1}$.

There is a category $\F_R$-$\mathrm{Mod}$ of $\F$-modules over $R$, where a morphism $(M, \theta_M) \rightarrow (N, \theta_N)$ is an $R$-linear map $f: M \rightarrow N$ such that $\theta_N \circ f = F^*f \circ \theta_M$. We denote by $\Hom_{\F}(M,N)$ the $\mathbb{F}_p$-space of $\F$-module morphisms $(M, \theta_M) \rightarrow (N, \theta_N)$. In particular, we can speak of \emph{$\F$-submodules}: an $\F$-submodule $N$ of an $\F$-module $(M, \theta_M)$ is an $R$-submodule $N \subseteq M$ such that the restriction $\theta_M|_N$ is an isomorphism $N \xrightarrow{\sim} F^*N$ (since $R$ is regular, $F^*$ is exact and so $F^*N$ can always be identified with an $R$-submodule of $F^*M$).  

In the fruitful analogy between $\D$-modules in characteristic zero and $\F$-modules in positive characteristic, the finiteness condition of holonomicity for $\D$-modules corresponds to the condition of ``$\F$-finiteness'' defined below (in particular, local cohomology modules provide examples of each). Loosely speaking, an $\F$-finite $\F$-module is one built from a finitely generated $R$-module by repeatedly applying the functor $F^*$ and passing to a colimit.

\begin{definition}\label{F-finite module}
Let $R$ be a regular Noetherian ring of characteristic $p > 0$, and let $(M, \theta_M)$ be an $\F$-module. We say that $M$ is \emph{$\F$-finite} if there exists a finitely generated $R$-module $M'$ and an $R$-linear map $\beta: M' \rightarrow F^*M'$ such that 
\[
\varinjlim(M' \xrightarrow{\beta} F^*M' \xrightarrow{F^*\beta} (F^*)^2 M' \to \cdots) \cong M,
\]
and the structure morphism $\theta_M$ is induced by taking the colimit over $l$ of $(F^*)^l\beta: (F^*)^l M' \rightarrow (F^*)^{l+1} M'$. In this case we call $M'$ a \emph{generator} of $M$ and $\beta$ a \emph{generating morphism}.
\end{definition}

\begin{example}\label{R, E, and local coh}
Let $R$ be a regular Noetherian ring of characteristic $p > 0$. The following are the most relevant examples of $\F$-finite $\F$-modules for the purposes of this paper.
\begin{enumerate}[(a)]
\item $R$ itself is an $\F$-finite $\F$-module. The corresponding Frobenius module structure is given by $\varphi_R = F_R$, and $\mathrm{id}_R$ is an $\F$-module generating morphism for $R$. Moreover, $R$ is a simple $\F$-module, since $\F$-submodules of $R$ are ideals $I \subseteq R$ such that the natural surjection $R/I^{[p]} \rightarrow R/I$ is an isomorphism (here $I^{[p]}$ is the ideal generated by all $p$th powers of elements of $I$), and as $R$ is Noetherian, this can only happen if $I = (0)$ or $I = R$.
\item If $I \subseteq R$ is an ideal and $i \geq 0$, the local cohomology module $H^i_I(R)$ is an $\F$-finite $\F$-module \cite[Example 2.2(b)]{LyubeznikFMod}.
\item If $\fm \subseteq R$ is a maximal ideal, then the $R$-injective hull $E = E_R(R/\fm)$ of the $R$-module $R/\fm$ is an $\F$-finite $\F$-module \cite[Proposition 5.4(d)]{SwitalaZhangInjDim}. Moreover, $E$ is a simple $\F$-module, since $R/\fm$ is a simple $R$-module as well as an $\F$-module generator of $E$.
\end{enumerate}
\end{example}





\begin{prop}\label{Matlis dual commutes with F}
Let $R$ be a regular Noetherian ring of characteristic $p > 0$, and let $J$ be an injective $R$-module. Denote by $D_J(-)$ the exact functor $\Hom_R(-,J)$ on the category of $R$-modules. There are $R$-module homomorphisms
\[
\delta_M: F^*D_J(M) \rightarrow D_J(F^*M)
\]
for all $R$-modules $M$, functorial in $M$. Furthermore, if $M$ is a finitely generated $R$-module, then the $\delta_M$ is an isomorphism. 
\end{prop}

\begin{proof}
There is a functorial $R$-module homomorphism
\[F^*D_J(M)=F_*R\otimes_R \Hom_R(M,J)\xrightarrow{\delta_1}\Hom_R(F_*R\otimes_RM, F_*R\otimes_RJ)\]
defined via 
\[(\delta_1(r\otimes \varphi))(s\otimes m)=rs\otimes \varphi(m)\]
Since $R$ is regular, $F_*R$ is a flat $R$-module. Hence, $\delta_1$ is an isomorphism when $M$ is finitely generated.

By \cite[Proposition 1.5]{HunekeSharp}, since $R$ is Gorenstein, $J \cong F^*J$ as $R$-modules. Fix a choice of $R$-module isomorphism $\theta^{-1}_J: F_*R\otimes_RJ \rightarrow J$ which induces an isomorphism 
\[\widetilde{\theta^{-1}}_J: \Hom_R(F_*R\otimes_RM, F_*R\otimes_RJ)\xrightarrow{\sim}\Hom_R(F_*R\otimes_RM, J)=D_J(F^*M)\]
. For $M$ an $R$-module, we define 
\[\delta_M=\widetilde{\theta^{-1}}_J\circ \delta_1.\]
Then $\delta_M: F^*D_J(M)\to D_J(F^*M)$ is an isomorphism when $M$ is finitely generated.
\end{proof}

Given two $\F$-modules over $R$ (say $M$ and $N$) we can regard them as Frobenius modules, and consider the $\mathbb{F}_p$-space of $\F$-module (resp. Frobenius module) morphisms between them. We show in Proposition \ref{Frobenius structure on Hom} below that not only are these two sets of morphisms the same, but that this set arises as the fixed part of a certain Frobenius module structure on $\Hom_R(M,N)$ itself, explained next. 

\begin{remark}\label{Remark: Frobenius structure on Hom}
Let $R$ be a regular Noetherian ring of characteristic $p > 0$. Let $(M, \theta_M)$ and $(N, \theta_N)$ be $\F$-modules over $R$. Then $\Hom_R(M,N)$ admits a natural Frobenius module structure as follows. Define $\varphi: \Hom_R(M,N) \rightarrow F_*\Hom_R(M,N)$ by 
\[
\varphi(f) = \theta_N^{-1} \circ (\mathrm{id}_R \otimes f) \circ \theta_M
\]
for each $f \in \Hom_R(M,N)$. It is clear that $\varphi$ is additive; it remains to show it is $R$-linear. Given any $r \in R$, we have
\[
\varphi(rf) = \theta_N^{-1} \circ (\mathrm{id}_R \otimes rf) \circ \theta_M = \theta_N^{-1} \circ (\mu_{r^p} \otimes f) \circ \theta_M = \mu_{r^p} \circ \theta_N^{-1} \circ (\mathrm{id}_R \otimes f) \circ \theta_M = r^p\varphi(f) = r\ast \varphi(f)
\]
for all $f \in \Hom_R(M,N)$, where $\mu_s$ (for any $s \in R$) denotes multiplication by $s$. It follows that $\varphi$ is $R$-linear, and hence it provides a Frobenius module structure on $\Hom_R(M,N)$.
\end{remark}

\begin{prop}\label{Frobenius structure on Hom}
Let $R$ be a regular Noetherian ring of characteristic $p > 0$. Let $(M, \theta_M)$ and $(N, \theta_N)$ be $\F$-modules over $R$. Regard $M$ (resp. $N$) as a Frobenius module via the $R$-linear map $\varphi_M: M \rightarrow F_*M$ (resp. $\varphi_N: N \rightarrow F_*N$) corresponding via adjunction to $\theta_M^{-1}$ (resp. $\theta_N^{-1}$). Then
\[
\Hom_{R[F]}(M,N) = \Hom_{\F}(M,N) = \Hom_R(M,N)^{\varphi = 1}
\]
as $\mathbb{F}_p$-subspaces of $\Hom_R(M,N)$, where the Frobenius module structure $\varphi$ on $\Hom_R(M,N)$ is defined as in Remark \ref{Remark: Frobenius structure on Hom}.
\end{prop}

\begin{proof}
The first equality has nothing to do with the choice of $\varphi$. Let $f \in \Hom_R(M,N)$ be given. On the one hand, the map $f$ belongs to $\Hom_{R[F]}(M,N)$ if and only if $\varphi_N \circ f = F_*f \circ \varphi_M$. On the other hand, $f$ belongs to $\Hom_{\F}(M,N)$ if and only if $\theta_N \circ f = F^*f \circ \theta_M$; equivalently, $f \circ \theta_M^{-1} = \theta_N^{-1} \circ F^*f$. We have
\[
f(\theta_M^{-1}(r \otimes m)) = f(r\varphi_M(m)) = rf(\varphi_M(m))
\]
and 
\[
\theta_N^{-1}(F^*f(r \otimes m)) = \theta_N^{-1}(r \otimes f(m)) = r\varphi_N(f(m))
\]
for all $r \in R$ and $m \in M$, and so the equality $f \circ \theta_M^{-1} = \theta_N^{-1} \circ F^*f$ is equivalent to $\varphi_N \circ f = F_*f \circ \varphi_M$. Thus $\Hom_{R[F]}(M,N) = \Hom_{\F}(M,N)$. 

For the second equality, observe that a map $f \in \Hom_R(M,N)$ belongs to the fixed part $\Hom_R(M,N)^{\varphi = 1}$ if and only if $f = \theta_N^{-1} \circ (\mathrm{id}_R \otimes f) \circ \theta_M$, or equivalently, $\theta_N \circ f = F^*f \circ \theta_M$. This is exactly the criterion for $f$ to be an $\F$-module morphism, completing the proof.
\end{proof}

\begin{remark}\label{special case of R}
Let $R$ be a regular Noetherian ring of characteristic $p>0$, and let $(M,\theta_M)$ be an $\F$-module over $R$. By Example \ref{R, E, and local coh}(a), the $R$-module $R$ itself is an $\F$-module with corresponding Frobenius module structure given by $\varphi_R = F_R$. Under the canonical identification of $M$ with $\Hom_R(R,M)$, the Frobenius module structure on $\Hom_R(R,M)$ provided by Remark \ref{Remark: Frobenius structure on Hom} coincides with the Frobenius module structure $\varphi_M$ corresponding by adjunction to the given $\F$-module structure on $M$ itself. Indeed, if $f: R \rightarrow M$ is defined by $f(1) = m$, then the composite $R \xrightarrow{\theta_R} F^*R \xrightarrow{\mathrm{id}_R \otimes f} F^*M \xrightarrow{\theta_M^{-1}} M$ maps $1 \mapsto 1 \otimes 1 \mapsto 1 \otimes m \mapsto \varphi_M(m)$.
\end{remark}

\section{Proof of the main theorem}\label{mainthm}

\begin{lem}\label{Frob structure on dual}
Let $R$ be a regular Noetherian ring of characteristic $p > 0$, and let $J$ be an injective $R$-module. Denote by $D_J(-)$ the exact functor $\Hom_R(-,J)$ on the category of $R$-modules. If $M$ is an $R$-module, then any $R$-module homomorphism $M \rightarrow F^*M$ induces a Frobenius module structure on $D_J(M)$.
\end{lem}

\begin{proof}
Apply the functor $D_J$ to the given map, obtaining an $R$-linear map $D_J(F^*M) \rightarrow D_J(M)$. Pre-composition with the map $\delta_M$ defined in Proposition \ref{Matlis dual commutes with F} gives an $R$-linear map $F^*D_J(M) \rightarrow D_J(M)$, which corresponds by adjunction (Proposition \ref{adjunction}) to an $R$-linear map $D_J(M) \rightarrow F_*D_J(M)$, the desired Frobenius module structure.
\end{proof}

In particular, if $M$ is an $\F$-module (resp. a generator of an $\F$-finite $\F$-module), then for any injective $R$-module $J$, $D_J(M)$ has a Frobenius module structure obtained by applying Lemma \ref{Frob structure on dual} to the structure morphism $\theta_M$ (resp. the generating morphism $\beta$).

\begin{definition}\label{F-finite}
Let $A$ be a ring of characteristic $p > 0$. We say that $A$ is \emph{$F$-finite}\footnote{Note the distinction between \emph{$F$-finiteness}, a property of a ring, and \emph{$\F$-finiteness}, a property of an $\F$-module.} if $F_*A$ is a finitely generated $A$-module.
\end{definition}

For example, if $k$ is a perfect field of characteristic $p > 0$, the rings $k[x_1, \ldots, x_n]$ and $k[[x_1, \ldots, x_n]]$ are $F$-finite. We recall the well-known facts that if $R$ is an $F$-finite regular Noetherian ring of characteristic $p > 0$, then $R$ is \emph{$F$-split} (meaning that the $R$-module homomorphism $R \rightarrow F_*R$ defined by the Frobenius admits a section) and $F_*R$ is locally free as an $R$-module (because it is finitely generated as well as flat).

\begin{thm}\label{main theorem on E}
Let $R$ be an $F$-finite regular Noetherian ring containing an algebraically closed field $k$ of characteristic $p > 0$, and let $J$ be an injective $R$-module. Denote by $D_J(-)$ the exact functor $\Hom_R(-,J)$ on the category of $R$-modules. Assume that the following conditions are satisfied:
\begin{enumerate}[(i)]
\item $J$, which is an $\F$-module by \cite[Proposition 1.5]{HunekeSharp}, is simple as an $\F$-module;
\item for every finitely generated $R$-module $M'$ equipped with a choice of $R$-module homomorphism $M' \rightarrow F^*M'$, the stable part $D_J(M')_s$ (which is defined by Lemma \ref{Frob structure on dual}) is a finite-dimensional $k$-space, and the Frobenius structure on $D_J(M')$ restricts to a bijection on $D_J(M')_s$.
\end{enumerate} 
Then, for each $\F$-finite $\F$-module $M$, the following numbers are all equal (and, in particular, are all finite):
\begin{enumerate}
\item the $\mathbb{F}_p$-dimension of $\Hom_{R[F]}(M,J) = \Hom_{\F}(M,J)$,
\item the $\mathbb{F}_p$-dimension of $D_J(M)^{\varphi=1}$,
\item the $k$-dimension of $D_J(M)_s$,
\item the $k$-dimension of $D_J(M')_s$, where $M'$ is any $\F$-module generator of $M$,
\item the maximal integer $t$ such that there exists a surjective $\F$-module morphism (equivalently, surjective Frobenius module morphism) $M \rightarrow J^t$.
\end{enumerate}
\end{thm}

\begin{proof}
Let $M$ be an $\F$-finite $\F$-module over $R$, and let $\beta: M' \rightarrow F^*M'$ be an $\F$-module generating morphism for $M$. We have already proved the equality of (1) and (2) above, in Proposition \ref{Frobenius structure on Hom}. Since $M'$ is a finitely generated $R$-module, our condition (ii) implies that $D_J(M')_s$ is a finite-dimensional $k$-space, that is, that (4) is finite. If we can prove the equality of (3) and (4) and hence the finiteness of (3), then the equality of (2) and (3) will follow from Proposition \ref{fixed versus stable part}. Therefore we need only prove the equality of (1) and (5) as well as the equality of (3) and (4).

We begin with the equality of (1) and (5). Suppose first that there exists an $\F$-module surjection $M \rightarrow J^t$. Post-composing it with each of the $t$ coordinate projections $J^t \rightarrow J$ produces $t$ $\mathbb{F}_p$-linearly independent $\F$-module morphisms $M \rightarrow J$. 

Conversely, assume there are $t$ such $\mathbb{F}_p$-linearly independent $\F$-module morphisms $\varphi_1, \ldots, \varphi_t: M \rightarrow J$; we wish to construct an $\F$-module surjection $M \rightarrow J^t$, or equivalently, an $\F$-submodule $N \subseteq M$ such that $M/N$ is isomorphic to $J^t$. Since $J$ is a simple $\F$-module by our condition (i), each $\varphi_i$ must itself be surjective, since its image is a non-zero $\F$-submodule of $J$. Set $M_i = \ker(\varphi_i)$, an $\F$-submodule of $M$, for all $i$. Since $M/M_i \cong J$, we must have $M_i + M_j = M$ whenever $i \neq j$. We claim that $M/(\cap_{i=1}^t M_i) \cong J^t$, which will complete the proof (take $N = \cap_{i=1}^t M_i$); we do this by showing, by induction on $j$, that $M/(\cap_{i=1}^j M_i) \cong J^j$ for $1 \leq j \leq t$. This assertion is obvious for $j=1$. Now suppose that for some $1 \leq j < t$ we know that $M/(\cap_{i=1}^j M_i) \cong J^j$ as $\F$-modules. We cannot have $\cap_{i=1}^j M_i \subseteq M_{j+1}$, since otherwise $\varphi_{j+1}$ would factor through $M \rightarrow M/(\cap_{i=1}^j M_i)$ and hence would lie in the $\mathbb{F}_p$-span of $\varphi_1, \ldots, \varphi_j$, a contradiction. Therefore $\cap_{i=1}^j M_i \nsubseteq M_{j+1}$, and so $\cap_{i=1}^j M_i + M_{j+1} = M$ by the simplicity. But then the short exact sequence
\[
0 \rightarrow M/(\cap_{i=1}^{j+1} M_i) \rightarrow M/(\cap_{i=1}^j M_i) \oplus M/M_{j+1} \rightarrow M/(M_{j+1} + \cap_{i=1}^j M_i) \rightarrow 0
\] 
of $\F$-modules implies that $M/(\cap_{i=1}^{j+1} M_i) \cong M/(\cap_{i=1}^j M_i) \oplus M/M_{j+1} \cong J^j \oplus J \cong J^{j+1}$ by the induction hypothesis, as desired.

Finally, we prove the equality of (3) and (4). By definition,
\[
M = \varinjlim(M' \xrightarrow{\beta} F^*M' \xrightarrow{F^*\beta} (F^*)^2 M' \rightarrow \cdots)
\]
and the $\F$-module structure on $M$ is induced by $\beta$ and its $F^*$-iterates. Applying $D_J(-)$, we find
\[
D_J(M) \cong \varprojlim(\cdots \rightarrow D_J((F^*)^2 M') \xrightarrow{D_J(F^*\beta)} D_J(F^* M') \xrightarrow{D(\beta)} D_J(M')).
\]
Since $R$ is $F$-finite, not only $M'$ but also $(F^*)^l M'$ for all $l \geq 0$ are finitely generated $R$-modules. Therefore, using Proposition \ref{Matlis dual commutes with F} to identify $D((F^*)^l M')$ with $(F^*)^lD_J(M')$ for all $l \geq 0$, we can rewrite the limit as
\[
D_J(M) \cong \varprojlim(\cdots \rightarrow (F^*)^2 D_J(M') \rightarrow F^*D_J(M') \rightarrow D_J(M')).
\]
Since $(F^*)^l$ and $(F^l)^*$ are isomorphic functors, this is exactly the \emph{leveling functor} of \cite[p. 47]{HartshorneSpeiser}. (In the notation of \cite{HartshorneSpeiser}, we have $D_J(M) = G(D_J(M'))$.) It follows from the proof of \cite[Proposition 1.2(b)]{HartshorneSpeiser} (see Remark \ref{HS proof} below) that 
\[
D_J(M)_s \cong \varprojlim(\cdots \rightarrow k \otimes_{F_k^2} D_J(M')_s \rightarrow k \otimes_{F_k} D_J(M')_s \rightarrow D_J(M')_s),
\]
where $F_k: k \rightarrow k$ is the Frobenius endomorphism of $k$, and the maps 
\[
k \otimes_{F_k^{l+1}} D_J(M')_s \rightarrow k \otimes_{F_k^l} D_J(M')_s
\] 
are given by the identity on the first tensor factor and the restriction of the map $D_J(M') \rightarrow D_J(M')$ defining the Frobenius module structure on $D_J(M')$ in the second. But by our condition (ii), this last map restricts to a bijection from $D_J(M')_s$ to itself. That is, the displayed limit can be identified with $D_J(M')_s$, so that $D_J(M)_s \cong D_J(M')_s$ as $k$-spaces, completing the proof of the equality of (3) and (4) and therefore the proof of the theorem.
\end{proof}

\begin{remark}\label{HS proof}
In the proof of Theorem \ref{main theorem on E} above, we appealed to \cite[Proposition 1.2(b)]{HartshorneSpeiser}. This proposition is stated in \cite{HartshorneSpeiser} only for a ring $R$ of characteristic $p > 0$ such that $F_*R$ is a \emph{free} $R$-module, a hypothesis that is stronger than ours. However, examining the proof of \cite[Proposition 1.2(b)]{HartshorneSpeiser}, it is clear that this hypothesis is only used in the form of the following consequence: if $M$ is an $R$-module and $m, m' \in M$ are such that $1 \otimes m = 1 \otimes m'$ in $F^*M = F_*R \otimes_{R} M$, then $m = m'$. But since we assumed in Theorem \ref{main theorem on E} that $R$ is regular and $F$-finite, it is also $F$-split, from which the previous statement is immediate.
\end{remark}

The following corollary of Theorem \ref{main theorem on E}, which identifies a class of injective modules for which the hypotheses of the theorem are satisfied, is the main result of this paper.

\begin{cor}\label{special case of main theorem}
Let $(R,\fm)$ be a Noetherian regular local ring containing its algebraically closed residue field $k$ of characteristic $p > 0$, and let $M$ be an $\F$-finite $\F$-module over $R$. Let $E = E_R(R/\fm)$ be the $R$-injective hull of $R/\fm$, and denote by $D(-)$ the exact functor $\Hom_R(-,E)$ on the category of $R$-modules. Then, for each $\F$-finite $\F$-module $M$, the following numbers are all equal (and, in particular, are all finite):
\begin{enumerate}
\item the $\mathbb{F}_p$-dimension of $\Hom_{R[F]}(M,E) = \Hom_{\F}(M,E)$,
\item the $\mathbb{F}_p$-dimension of $D(M)^{\varphi=1}$,
\item the $k$-dimension of $D(M)_s$,
\item the $k$-dimension of $D(M')_s$, where $M'$ is any $\F$-module generator of $M$,
\item the maximal integer $t$ such that there exists a surjective $\F$-module morphism (equivalently, surjective Frobenius module morphism) $M \rightarrow E^t$.
\end{enumerate}
\end{cor}

\begin{proof}
By Example \ref{R, E, and local coh}(c), $E$ is a simple $\F$-module, so condition (i) of Theorem \ref{main theorem on E} is satisfied. Now let $M'$ be a finitely generated $R$-module equipped with a choice of $R$-module homomorphism $M' \rightarrow F^*M'$. There exists an $R$-linear surjection $R^l \rightarrow M'$ for some $l \geq 0$; applying the exact functor $D$, we obtain an $R$-linear injection $D(M') \rightarrow D(R^l) = E^l$. Since $E$ (and hence $E^l$) is an Artinian $R$-module supported only at $\fm$, the same is true of $D(M')$, so $D(M')$ has a natural structure as a module over the $\fm$-adic completion $\widehat{R}^{\fm}$ of $R$. In fact, $D(M')$ is a Frobenius module over $\widehat{R}^{\fm}$, with the Frobenius structure given by the same underlying additive map $D(M') \rightarrow D(M')$ defined by Lemma \ref{Frob structure on dual}. The ring $\widehat{R}^{\fm}$ is local (since $\fm$ is maximal) and contains an algebraically closed (hence perfect) field $k$; moreover, $D(M')$ is Artinian as an $\widehat{R}^{\fm}$-module. Therefore, Theorem \ref{Artinian stable part finite} applies. We conclude that $D(M')_s$ is a finite-dimensional $k$-space and the Frobenius structure on $D(M')$ restricts to a bijection on $D(M')_s$, so condition (ii) of Theorem \ref{main theorem on E} is satisfied. The corollary now follows from Theorem \ref{main theorem on E} applied to $J = E$.
\end{proof}


The proof of the equality of (1) and (5) in Corollary \ref{special case of main theorem} works in characteristic zero as well, replacing ``$\F$-finite $\F$-module'' with ``holonomic $\D$-module''. Therefore we obtain an alternate proof of the fact that if $R = k[x_1, \ldots, x_n]$ or $k[[x_1, \ldots, x_n]]$ where $k$ is a field of characteristic zero, and $M$ is a holonomic $\D(R,k)$-module, then $\dim_k \Hom_{\D}(M,E)$ is equal to the maximal integer $t$ for which there exists a $\D$-linear surjection $M \rightarrow E^t$. This statement is part of \cite[Corollary 5.2]{HartshornePolini}. An easier ``dual'' statement is the following \cite[Lemma 2.3]{SwitalaZhangGradedDual}: $\dim_k \Hom_{\D}(R,M)$ is equal to the maximal integer $t$ for which there exists a $\D$-linear injection $R^t \rightarrow M$. We can prove a version of this in the Frobenius module setting, as part of a ``dual'' version of Corollary \ref{special case of main theorem}. Note, however, that Theorem \ref{main theorem on R} has a finite-dimensionality hypothesis whose analogue is not needed (because it is automatically satisfied) in Corollary \ref{special case of main theorem}. 

\begin{thm}\label{main theorem on R}
Let $R$ be a regular Noetherian ring containing an algebraically closed field $k$ of characteristic $p > 0$, and let $M$ be an $\F$-finite $\F$-module over $R$. Assume that $M_s$ is a finite-dimensional $k$-space. Then the following numbers are all equal (and, in particular, are all finite):
\begin{enumerate}
\item the $\mathbb{F}_p$-dimension of $\Hom_{R[F]}(R,M) = \Hom_{\F}(R,M)$,
\item the $\mathbb{F}_p$-dimension of $M^{\varphi=1}$,
\item the $k$-dimension of $M_s$,
\item the maximal integer $t$ such that there exists an injective $\F$-module morphism (equivalently, injective Frobenius module morphism) $R^t \rightarrow M$.
\end{enumerate}
\end{thm}

\begin{proof}
By Remark \ref{special case of R}, we can identify $M$ with $\Hom_R(R,M)$ as Frobenius modules over $R$, and therefore we can identify $M^{\varphi=1}$ with $\Hom_R(R,M)^{\varphi=1}$ as $\mathbb{F}_p$-spaces. Therefore the equality of (1) and (2) follows from Proposition \ref{Frobenius structure on Hom}. We have assumed that $M_s$ is a finite-dimensional $k$-space, so the equality of (2) and (3) follows from Proposition \ref{fixed versus stable part}. Finally, the proof of the equality of (1) and (4) is essentially dual to the proof of the equality of (1) and (5) in Theorem \ref{main theorem on E}, using the fact that $R$ is a simple $\F$-module (Example \ref{R, E, and local coh}(a)). The arguments are similar enough that we omit the details, providing a sketch. An $\F$-module injection $R^t \rightarrow M$ gives rise to $t$ $\mathbb{F}_p$-linearly independent $\F$-module morphisms $R \rightarrow M$ by pre-composition with the coordinate inclusions; conversely, given $t$ distinct isomorphic copies of $R$ (say $M_1, \ldots, M_t$) as $\F$-submodules of $M$, it can be shown (since all $M_i$ are simple $\F$-submodules) that the sum $\sum_{i=1}^j M_i \subseteq M$ is a \emph{direct} sum for $j = 1, \ldots, t$ by induction on $j$, and the case $j=t$ is the desired assertion.
\end{proof}

\begin{question}\label{finite stable part of F-finite}
If $R$ is a regular Noetherian ring containing an algebraically closed field $k$ of characteristic $p > 0$, and $M$ is an $\F$-finite $\F$-module over $R$, is $M_s$ a finite-dimensional $k$-space?
\end{question}

A positive answer to Question \ref{finite stable part of F-finite} would, of course, permit us to remove the finite-dimensionality hypothesis in Theorem \ref{main theorem on R}, since Proposition \ref{fixed versus stable part} would apply.

We end with the following remarks.

\begin{remark}
\label{remark: recover etale cohomology}
Let $X\subset \mathbb{P}^n_k$ be a projective scheme over an algebraically closed field $k$ of characteristic $p$ and let $I\subset R=k[x_0,\dots,x_n]$ be its defining ideal. Applying Corollary \ref{special case of main theorem} to the $\F$-finite $\F$-module $H^j_I(R)$, one has
\begin{align*}
\dim_{\mathbb{F}_p} \Hom_{\F}(H^{n-j}_I(R),E) &= \dim_k \Hom_R(H^{n-j}_I(R),E)_s \\
&= \dim_k \Hom_R(\Ext^{n-j}_R(R/I,R),E)_s \\
&= \dim_k H^{j+1}_{\fm}(R/I)_s \\
&= \dim_k H^j(X, \mathscr{O}_X)_s
\end{align*}
According to \cite[proposition 5.1]{HartshorneSpeiser}, $\dim_k H^j(X, \mathscr{O}_X)_s= \dim_{\mathbb{F}_p} H^j_{\mathrm{\acute{e}t}}(X,\mathbb{F}_p)$. Hence 
\[\dim_{\mathbb{F}_p} \Hom_{\F}(H^{n-j}_I(R),E)= \dim_{\mathbb{F}_p} H^j_{\mathrm{\acute{e}t}}(X,\mathbb{F}_p).\]

Assume that the local cohomology module $H^{n-j}_I(R)$ is supported only at $\fm$ ({\it e.g.} when $X$ is Cohen-Macaulay). Then using the previous paragraph one may recover \cite[Corollary 2.4]{HartshorneSpeiser} as follows:
\[H^{n-j}_I(R) \cong E^{\lambda_j},\ {\rm where}\ \lambda_j=\dim_{\mathbb{F}_p}(H^j_{\mathrm{\acute{e}t}}(X,\mathbb{F}_p)).\]
\end{remark}

\begin{remark}
Let $X,R,I$ be the same as in Remark \ref{remark: recover etale cohomology} and, additionally, assume that $X$ is irreducible and each local ring $\mathscr{O}_{X,x}$ is $F$-rational for each point $x\in X$. Our Remark \ref{remark: recover etale cohomology} implies that $H^{n-\dim(X)}_I(R)$ admits a $\F$-module quotient that is isomorphic to $E^{\lambda_d}$ where $\lambda_d=\dim_{\mathbb{F}_p}(H^{\dim(X)}_{\mathrm{\acute{e}t}}(X,\mathbb{F}_p))$. This does not fully recover the prime-characteristic analogue of \cite[Theorem 6.4]{HartshornePolini} in the case when $i=n-\dim(X)$; the length differs by one from the desired analogous result. The reason is that the simple $\F$-submodule of $H^{n-\dim(X)}_I(R)$ does not admit any non-zero $\F$-module morphism to $E$. However, it follows directly from \cite[Theorem 4.3]{KatzmanMaSmirnovZhang} that $H^{n-\dim(X)}_I(R)$ admits a simple $\F$-submodule $H_0$ such that $H^{n-\dim(X)}_I(R)/H_0\cong E^{\lambda_d}$. 
\end{remark}


\bibliographystyle{plain}

\bibliography{masterbib}

\end{document}